\documentclass{proc-l}
\usepackage[cmtip,all]{xy}
\usepackage{amssymb,mathrsfs,enumitem,amsmath,bbold}
\usepackage[mathscr]{euscript}
\usepackage{xcolor,fourier,stackrel}
\definecolor{Chocolat}{rgb}{0.36, 0.2, 0.09}
\definecolor{BleuTresFonce}{rgb}{0.215, 0.215, 0.36}
\definecolor{EgyptianBlue}{rgb}{0.06, 0.2, 0.65}

\usepackage[colorlinks,final,hyperindex]{hyperref}
\hypersetup{citecolor=BleuTresFonce, urlcolor=EgyptianBlue, linkcolor=Chocolat}

\newtheorem{theorem}{Theorem}[section]
\newtheorem{corollary}[theorem]{Corollary}

\theoremstyle{definition}

\DeclareMathAlphabet{\pazocal}{OMS}{zplm}{m}{n}

\def\calB{\pazocal{B}}
\def\calC{\pazocal{C}}

\def\calP{\pazocal{P}}
\def\calQ{\pazocal{Q}}
\def\calR{\pazocal{R}}
\def\calS{\pazocal{S}}
\def\calT{\pazocal{T}}

\def\calX{\pazocal{X}}
\def\calY{\pazocal{Y}}

\DeclareMathOperator{\Hom}{Hom}
\DeclareMathOperator{\sgn}{sgn}

\DeclareMathOperator{\Lie}{\mathsf{Lie}}
\DeclareMathOperator{\Ass}{\mathsf{Ass}}

\DeclareMathOperator{\PL}{\mathsf{PreLie}}
\DeclareMathOperator{\rPL}{\mathsf{rPL}}
\DeclareMathOperator{\RT}{\mathsf{RT}}

\DeclareMathOperator{\Com}{\mathsf{Com}}
\DeclareMathOperator{\Perm}{\mathsf{Perm}}
\DeclareMathOperator{\Tw}{Tw}
\DeclareMathOperator{\MC}{MC}
\DeclareMathOperator{\In}{in}

\DeclareMathOperator{\Br}{\mathsf{Brace}}
\DeclareMathOperator{\Gerst}{\mathsf{Gerst}}

\DeclareMathOperator{\uCom}{\mathsf{uCom}}
\DeclareMathOperator{\uAss}{\mathsf{uAss}}

\newcommand{\Def}{\mathrm{Def}}

\DeclareMathAlphabet{\mathbbold}{U}{bbold}{m}{n}

\def\k{\mathbbold{k}}

\newcommand{\ac}{\scriptstyle \text{\rm !`}}

\begin{document}

\title{Homotopical rigidity of the pre-Lie operad}

\author{Vladimir Dotsenko}
\address{Institut de Recherche Math\'ematique Avanc\'ee, UMR 7501, Universit\'e de Strasbourg et CNRS, 7 rue Ren\'e-Descartes, 67000 Strasbourg CEDEX, France}
\email{vdotsenko@unistra.fr}

\author{Anton Khoroshkin}
\address{National Research University Higher School of Economics,
20 Myasnitskaya street, Moscow 101000, Russia \&
Institute for Theoretical and Experimental Physics, Moscow 117259, Russia}
\email{akhoroshkin@hse.ru}

\dedicatory{This paper is dedicated to Martin Markl on the occasion of his sixtieth birthday}

\begin{abstract}
We show that the celebrated operad of pre-Lie algebras is very rigid: it has no  ``non-obvious'' degrees of freedom from either of the three points of view: deformations of maps to and from the ``three graces of operad theory'', homotopy automorphisms, and operadic twisting. Examining the latter, it is possible to answer two questions of Markl from 2005, including a Lie-theoretic version of the Deligne conjecture.
\end{abstract}

\maketitle

\section*{Introduction} 
Recall that a pre-Lie algebra is a vector space equipped with a bilinear operation $a,b\mapsto a\triangleleft b$ satisfying the identity  
 \[
(a_1\triangleleft a_2)\triangleleft a_3 - a_1\triangleleft (a_2\triangleleft a_3) = (a_1\triangleleft a_3)\triangleleft a_2- a_1\triangleleft (a_3\triangleleft a_2) ,
 \]
known as the pre-Lie identity or the right-symmetric identity. Pre-Lie algebras appear in a wide range of contexts across pure and applied mathematics, from algebra and combinatorics to differential geometry, numerical methods, and theory of renormalisation. In this paper, we study the operad of pre-Lie algebras from the homotopy theory viewpoint. 

It is almost immediate from the definition that each pre-Lie algebra may be regarded as a Lie algebra with the bracket $[a,b]:=a\triangleleft b-b \triangleleft a$, each associative algebra may be regarded as a pre-Lie algebra, and, of course, each commutative associative algebra may be regarded as a pre-Lie algebra. These facts correspond to statements on the level of operads: namely, there are maps of the corresponding operads which induce these functors between categories of algebras.  The first goal of this paper is to show that deformation complexes of these maps to and from the operad $\PL$ are acyclic. Moreover, the deformation complex of the identity map of the operad $\PL$ is also acyclic, implying that the only homotopy automorphisms of the operad $\PL$ are the intrinsic ones given by re-scaling. Our second goal is to study the result of applying to the operad $\PL$ the general construction of operadic twisting due to Thomas Willwacher \cite{MR3348138}. We compute the homology of the operad $\Tw(\PL)$, showing that it coincides with the operad~$\Lie$. Our main motivation to study the operad $\Tw(\PL)$ was coming from search of operations that are naturally defined on deformation complexes of maps of operads. Our theorem shows that the only homotopy invariant structure one may define functorially starting from the convolution pre-Lie algebra structure is that of a dg Lie algebra. 

Both results of this paper are strongly connected to Martin Markl's work on deformation theory. Deformation theory of maps of operads was developed by van der Laan \cite{vdLaan} drawing as an inspiration from Markl's work on models for operads and the cotangent cohomology \cite{MR1415558,MR1380606}. Moreover, it appears that the first time the deformation complex of the identity map of a Koszul operad received substantial attention was in paper \cite{MR2309979} where Markl proposed a general framework for studying natural operations on homology of deformation complexes; in that paper, the deformation complex of the identity map was christened the ``soul of the cohomology of $\calP$-algebras''. (A perhaps unfortunate consequence of this terminology is that cohomology of algebras over many important operads ends up being soulless.) That paper also features Markl's version of the Deligne conjecture for the operad $\Lie$ \cite[Conj.~7]{MR2309979} stating that the shifted Lie algebra structure is the only natural homotopy invariant algebraic structure defined on cotangent complexes of Lie algebras; this led to a sequel paper \cite{MR2325698} where Markl gave a beautiful but mysterious cohomological description of Lie elements in free pre-Lie algebras. In fact, we found a way to de-mystify that description: the cryptic construction of the operad $\rPL$ in \cite{MR2325698} is best understood in the context of the operadic twisting, as a layer of the operad $\Tw(\PL)$, answering Markl's question on higher homology of~$\rPL$. Our approach also brought us to a realisation that an earlier result of Willwacher \cite[Th.~3.6]{willwacher2017prelie} essentially resolves Markl's version of the Deligne conjecture for deformations of Lie algebras. In view of all these connections, it is only natural that we wish to dedicate our work to Martin for his birthday, wishing him many happy returns of the day. 

Our arguments use filtration arguments that appear quite often when working with graph complexes \cite{MR3348138,willwacher2017prelie}, so in addition to furnishing the proofs of several new results, this paper hopefully might serve a pedagogical purpose, giving the reader an insight into several useful tricks that are normally hidden deep inside very condensed papers. 

For all the relevant definition from the operad theory we refer the reader to the monograph \cite{MR2954392}. All vector spaces in this paper are defined over a field $\k$ of characteristic zero, and all chain complexes are homological (with the differential of degree $-1$). We use the symbol $s$ to handle suspensions, and the symbol $\calS$ for operadic suspensions. When writing down elements of operads, we use small latin letters as placeholders; for arguments of nontrivial homological degrees, there are extra signs which arise via the usual Koszul sign rule from applying operations to arguments.

\section{Combinatorics related to the pre-Lie operad}

Most of the arguments of this paper rely on the description of the operad $\PL$ controlling pre-Lie algebras in terms of labelled rooted trees due to Chapoton and Livernet~\cite{MR1827084} (which was perhaps known to Cayley \cite{Cayley1857}). In modern terms, the underlying $\mathbb{S}$-module of the operad $\PL$ is the linearisation of the species $\RT$ of labelled rooted trees, and the operadic insertion of trees can be described combinatorially as follows. For a labelled rooted tree~$S\in\RT(J)$ and a labelled rooted tree~$T\in\RT(I)$, the composition $T\circ_i S$ is equal to the sum
 \[
\sum_{f\colon \In_T(i)\to J} T\circ_i^f S ,
 \]
where the sum is over all functions $f$ from the set of incoming edges of the vertex labelled $i$ to the set~$J$ of all vertices of~$S$: the labelled rooted tree $T\circ_i^f S$ is obtained by grafting the tree $S$ in the place of the vertex~$i$, and grafting the subtrees growing from the vertex $i$ in $T$ at the vertices of $S$ according to the function~$f$, so that the set of incoming edges of each vertex~$j$ becomes $\In_S(j)\sqcup f^{-1}(j)$. For example, we have
 \[
\vcenter{
\xymatrix@M=3pt@R=5pt@C=5pt{
*+[o][F-]{1}\ar@{-}[dr] &&  *+[o][F-]{3}\ar@{-}[dl] \\
& *+[o][F-]{2} & 
}}
\ \circ_2 \ 
\vcenter{
\xymatrix@M=3pt@R=5pt@C=5pt{
*+[o][F-]{a}\ar@{-}[d]  \\
*+[o][F-]{c} 
}}=
\vcenter{
\xymatrix@M=3pt@R=5pt@C=5pt{
*+[o][F-]{1}\ar@{-}[dr] &*+[o][F-]{a}\ar@{-}[d] & *+[o][F-]{3}\ar@{-}[dl]\\
& *+[o][F-]{c} & 
}}+
\vcenter{
\xymatrix@M=3pt@R=5pt@C=5pt{
*+[o][F-]{1}\ar@{-}[dr] & & *+[o][F-]{3}\ar@{-}[dl]\\
 &  *+[o][F-]{a}\ar@{-}[d]& \\
& *+[o][F-]{c} & 
}}+
\vcenter{
\xymatrix@M=3pt@R=5pt@C=5pt{
  & *+[o][F-]{3}\ar@{-}[d]\\
*+[o][F-]{1}\ar@{-}[dr] &  *+[o][F-]{a}\ar@{-}[d]& \\
& *+[o][F-]{c} & 
}}+
\vcenter{
\xymatrix@M=3pt@R=5pt@C=5pt{
*+[o][F-]{1}\ar@{-}[d] & & \\
  *+[o][F-]{a}\ar@{-}[d]& *+[o][F-]{3}\ar@{-}[dl]\\
 *+[o][F-]{c} & 
}}\ .
 \]

The Koszul dual operad $\PL^!$ is also very easy to describe combinatoriallly. That operad is usually denoted $\Perm$, and the corresponding algebras are called permutative algebras; a permutative algebra is an associative algebra that additionally satisfies the identity $a_1a_2a_3=a_1a_3a_2$. It is clear that the underlying species of the operad $\Perm$ is the species of sets with one marked element (the first element of the $n$-fold product), so the corresponding representation of the symmetric group is the standard permutation representation. While permutative algebras themselves do not often arise naturally, the operad $\Perm$ is used in an important general construction of the operad theory going back to work of Chapoton \cite{MR1860996}: for any operad $\calP$, one may define the operad of $\calP$-dialgebras as the Hadamard product $\calP\stackrel[\mathrm{H}]{}{\otimes}\Perm$. Intuitively, one should think of $\calP$-dialgebras as of $\calP$-algebras with one element underlined. In some proofs of this paper, we shall encounter associative dialgebras \cite{MR1860994} and pre-Lie dialgebras \cite{felipe2011prelie,MR2419658}.

\section{Deformation theory for maps of Koszul operads}\label{sec:DefKoszul}

In this section, we assume each operad $\calP$ reduced ($\calP(0)=0$), connected ($\calP(1)\cong\k$), and with finite-dimensional components. We begin with briefly recalling a particular case of the general results of \cite{MR2572248,vdLaan} that highlight the role of pre-Lie algebras in deformation theory; these instances of pre-Lie algebras serve as a motivation for Section~\ref{sec:twisting} of this paper. Let $\calP$ be a Koszul operad, and let $f\colon \calP\to\calQ$ be a map from $\calP$ to a dg operad~$\calQ$. In this case, the general recipe for computing the deformation complex of the map $f$ produces a small and tractable chain complex in several easy steps. First, one should consider the \emph{convolution operad} between the Koszul dual cooperad of $\calP$ and the dg operad $\calQ$; its underlying dg $\mathbb{S}$-module is $\Hom(\calP^{\ac},\calQ)$, and the operad composition maps $\circ_i$ are computed using the general philosophy behind convolution products: to evaluate the operadic composition $\phi\circ_i\psi$ on $\alpha\in\calP^{\ac}$, one applies the cooperad decomposition map $\Delta_i$ of the cooperad $\calP^{\ac}$ to $\alpha$, computes the tensor product of maps $\phi\otimes\psi$ on the result, and computes the composition $\circ_i$ of the operad $\calQ$. As any operad, the convolution operad can be made into a pre-Lie algebra using the formula 
 \[
\phi\triangleleft\psi=\sum_i \phi\circ_i \psi.
 \]
One can check that invariants of symmetric groups are closed under this convolution product, so one may consider the following \emph{convolution pre-Lie algebra} 
 \[
\Hom^{\mathbb{S}}(\calP^{\ac},\calQ) := \prod_{n\ge 1}\Hom(\calP^{\ac}(n),\calQ(n))^{\mathbb{S}_n} .
 \]
Finally, using the Lie bracket $[a,b]:=a\triangleleft b-b \triangleleft a$ mentioned in the introduction, one may consider that space as a Lie algebra called \emph{the convolution (Lie) algebra}; it is a dg Lie algebra (with zero differential if the operad $\calQ$ has zero differential). Recall that the Maurer--Cartan equation in a dg Lie algebra $(L,[-,-],d)$ is the equation 
 \[
d(\alpha)+\frac12[\alpha,\alpha]=0,
 \] 
and elements $\alpha$ of degree $-1$ satisfying this condition are called Maurer--Cartan elements; it is possible to show that in our case of the convolution Lie algebra, Maurer--Cartan elements correspond to maps of operads from $\calP_\infty=\Omega(\calP^{\ac})$ to $\calQ$. In general, a Maurer--Cartan element in a Lie algebra can be used to twist the differential, letting
 \[
d_\alpha=d+[\alpha,-] .
 \]
We shall twist the differential in our dg Lie algebra using a particular Maurer--Cartan element $\alpha$ corresponding to the map from $\calP_\infty$ to $\calQ$ that is obtained from $f$ by the pre-composition with the projection $\calP_\infty\twoheadrightarrow\calP$. By definition, the \emph{deformation complex of map $f$} is the dg Lie algebra
 \[
\Def(f\colon\calP\to\calQ):=\left(\Hom^{\mathbb{S}}(\calP^{\ac},\calQ),[-,-],d_\alpha\right) .
 \] 
Maurer--Cartan elements of that differential graded Lie algebra, that is elements $\lambda$ of degree $-1$ satisfying 
 \[
d_\alpha(\lambda)+\frac12[\lambda,\lambda]=0,
 \]
correspond to infinitesimal deformations of the morphism $f$, and gauge equivalence of such corresponds to equivalence of deformations; thus, this complex essentially controls the deformation theory of the map $f$. In fact, one may replace the deformation complex by its homology with the transferred $L_\infty$-algebra structure: that $L_\infty$-algebra is filtered, so one may work with its Maurer--Cartan elements instead, not losing any information \cite{MR3431305,MR3323983}. 

We remark that it is common to remove the counit from the cooperad $\calP^{\ac}$ and define the deformation complex as 
$\left(\Hom^{\mathbb{S}}(\overline{\calP^{\ac}},\calQ),d_\alpha\right)$. Since our operads are assumed to be connected, the difference between the two complexes is a one-dimensional space, so it accounts just for one extra homology class. (If $\calQ=\calP$, the corresponding homology class accounts for the inner derivation of $\calP$ given by the commutator with the operad unit, or, after exponentiation, to re-scaling operations \cite[Rem.~6.3.3]{ksdeftheory}.) We prefer to work with the bigger complex $\Hom^{\mathbb{S}}(\calP^{\ac},\calQ)$, since it allows for more elegant results; however, it is important to remember that the abovementioned extra homology class always exists.

Let us record here a homology computation (due to Markl) for deformation complex of the identity map for each of the ``three graces of operad theory'' which is one of the earliest such computations in the literature\footnote{It would be fair to note that the argument of \cite[Ex.~14]{MR2309979} needs a minor correction: for $n\ge 3$ the $\mathbb{S}_n$-module $\Lie(n)$ does not contain the sign representation \cite{MR0371961}, so the deformation complexes for the identity maps of operads $\Com$ and $\Lie$ almost completely collapse even before passing to the homology.}; particular cases of these results (for the homology in degrees $1$ and $2$) have also recently been re-proved in \cite{bao2020cohomological}.

\begin{theorem}[{\cite[Th.~13 \& Ex.~14]{MR2309979}}]
The following complexes are acyclic:
\begin{enumerate}
\item the deformation complex of the identity map of the Lie operad,
\item the deformation complex of the identity map of the commutative operad, 
\item the deformation complex of the identity map of the associative operad. 
\end{enumerate}
\end{theorem}

Suppose that $\calP=\calT(\calX)/(\calR)$ and $\calQ=\calT(\calY)/(\calS)$ are two Koszul operads, and let $f\colon \calP\to\calQ$ be a map of operads induced by a map of quadratic data $(\calX,\calR)\to(\calY,\calS)$; we use the notation $f_\circ\colon\calX\to\calY$ for the corresponding map of generators. A map of quadratic data also induces a map of Koszul dual cooperads $f^{\ac}\colon\calP^{\ac}\to\calQ^{\ac}$. Moreover, under the finite-dimensionality assumptions imposed in the beginning of this section, one may dualise and take the Hadamard tensor product of the map $f_\circ^*\colon\calY^*\to\calX^*$ with $s^{-1}\calS^{-1}$ to obtain a well defined map of operads $f^!\colon\calQ^!\to\calP^!$. Unlike in the case of linear duality, it is not possible to predict what properties the Koszul dual map of a map of operads would have: 

-- the Koszul dual of the surjection $\PL\to\Com$ is the map $\Lie\to\Perm$ which is neither surjective nor injective (in fact, one can prove that the image of that map is isomorphic to the operad of Lie algebras satisfying the identity $[[a_1,a_2],[a_3,a_4]]=0$),

-- the Koszul dual of the embedding $\Lie\to\PL$ is the surjection $\Perm\to\Com$,

-- the Koszul dual of the surjection $\PL\to\Ass$ is the surjection $\Ass\to\Perm$. 

\noindent
However, deformation complexes behave well under Koszul duality, as we shall now show. 

\begin{theorem}\label{th:Koszul}
We have an isomorphism of differential graded Lie algebras
 \[
\Def(f\colon\calP\to\calQ)\cong\Def(f^!\colon\calQ^!\to\calP^!).
 \]
\end{theorem}

\begin{proof}
Throughout this proof, we shall assume for simplicity that our operads are generated by binary operations; the theorem holds in full generality, and we invite the reader to adapt the proof for the general case. The Maurer--Cartan element $\alpha$ corresponding to the map $f$ in the convolution Lie algebra $\Hom^{\mathbb{S}}(\calP^{\ac},\calQ)$ is  equal to $s^{-1} f_\circ$, where 
 \[
s^{-1} f_\circ\in s^{-1} \Hom^{\mathbb{S}}(\calX,\calY)\cong \Hom^{\mathbb{S}}(s\calX,\calY)\subset \Hom^{\mathbb{S}}(\calP^{\ac},\calQ).
 \]
Let us examine the formula for the deformation complex a bit closer. Using \cite[Sec.~7.2.3]{MR2954392}, we note that the underlying $\mathbb{S}$-module of the convolution operad $\Hom(\calP^{\ac},\calQ)$ is 
 \[
(\calP^{\ac})^*\stackrel[\mathrm{H}]{}{\otimes}\calQ \cong \calS\stackrel[\mathrm{H}]{}{\otimes}\calP^!\stackrel[\mathrm{H}]{}{\otimes}\calQ ,
 \]
with the operad structure given by the factor-wise operad structure on the Hadamard tensor product, and that the $\mathbb{S}$-module of generators of the Koszul dual operad $\calP^!$ is $s^{-1}\calS^{-1}\stackrel[\mathrm{H}]{}{\otimes}\calX^*$. In particular, the $\mathbb{S}$-submodule 
 $
s^{-1}\calX^*\stackrel[\mathrm{H}]{}{\otimes}\calQ\cong\Hom(s\calX,\calQ)$ of $\Hom(\calP^{\ac},\calQ)$ 
consisting of maps supported at the cogenerators of $\calP^{\ac}$ is identified with the submodule 
 \[
\calS\stackrel[\mathrm{H}]{}{\otimes}\left(s^{-1}\calS^{-1}\stackrel[\mathrm{H}]{}{\otimes}\calX^*\right)\stackrel[\mathrm{H}]{}{\otimes}\calQ
 \]
of the Hadamard product, and the space $s^{-1} \Hom^{\mathbb{S}}(\calX,\calY)$ is identified with 
 \[
\left(\calS(2)\otimes (s^{-1}\calS^{-1}(2)\otimes\calX^*(2))\otimes\calY(2)\right)^\mathbb{S} . 
 \]
Let us denote by $\mu$ the basis element $s^{-1}\in\k s^{-1}\cong\calS(2)$ and by $\nu$ the basis element $s\in\k s\cong\calS^{-1}(2)$. If we denote by $\{x_i\}$ a basis of $\calX(2)$ and by $\{x_i^\vee=s^{-1}\nu\otimes x_i^*\}$ the corresponding basis of $\calX^!(2)$, the element in the Hadamard product corresponding to the Maurer--Cartan element $\alpha$ used to twist the differential is
 $
\sum_i \mu\otimes x_i^\vee\otimes f(x_i) .
 $
Since the operad $\calQ$ is Koszul, we may use the same techniques for the map $f^!$. The corresponding convolution operad is
 \[
\calS\stackrel[\mathrm{H}]{}{\otimes}\calP^!\stackrel[\mathrm{H}]{}{\otimes}\calQ \cong\calS\stackrel[\mathrm{H}]{}{\otimes}\calQ\stackrel[\mathrm{H}]{}{\otimes}\calP^! \cong \calS\stackrel[\mathrm{H}]{}{\otimes}(\calQ^!)^!\stackrel[\mathrm{H}]{}{\otimes}\calP^!  .
 \]
We note that this operad is isomorphic to the convolution operad corresponding to the morphism~$f$. Consequently, the convolution Lie algebra on $\Hom^{\mathbb{S}}(\calP^{\ac},\calQ)$ is isomorphic to that on $\Hom^{\mathbb{S}}((\calQ^!)^{\ac},\calP^!)$. The canonical isomorphism 
 \[
\Hom^{\mathbb{S}}(\calX,\calY)\cong \Hom^{\mathbb{S}}(\calY^*,\calX^*)
 \]
sends $f_\circ$ to $f_\circ^*$, which easily implies that under our identifications the Maurer--Cartan elements corresponding to $f$ and $f^!$ are identified. This completes the proof.
\end{proof}

In a particular case where $\calP=\calQ$ and the map $f$ is the identity map, this theorem states that the deformation theory of the identity map is the same for a Koszul operad and its dual. In the case of associative algebras (i.e. operads concentrated in arity one), the deformation complex of the identity map of a Koszul algebra $A$ defined above has the same homotopy type as the Hochschild cohomology complex $C^{-\bullet}(A,A)$; for the case of operads it is an appropriate generalisation of the Hochschild complex. The fact that for a Koszul associative algebra the homotopy type of that differential graded Lie algebra is invariant under Koszul duality is due to Keller \cite{keller2004derived}. 

\section{Deformation complexes of map between the pre-Lie operad and the ``three graces''}\label{sec:maps}

In this section, we show that the maps between the operad $\PL$ and the operads $\Com$, $\Lie$ and~$\Ass$, christened by Jean--Louis Loday the ``three graces of operad theory'', are homotopically rigid. It is easy to show that there are, up to re-scaling, just three such maps mentioned in the introduction:
the projection from $\PL$ to $\Com$ sending the pre-Lie product to the product in the commutative operad, 
the map from $\Lie$ to $\PL$ sending $[a,b]$ to $a\triangleleft b-b \triangleleft a$, and 
the projection from $\PL$ to $\Ass$ sending the pre-Lie product to the associative product.

\subsection{The map to the commutative operad}

Our first vanishing theorem uses a simple representation-theoretic argument.  

\begin{theorem}\label{th:PLCom}
The deformation complex of the projection $\PL\to\Com$ is acyclic.
\end{theorem}

\begin{proof}
Here, one uses the convolution algebra 
$
\prod_{n\ge 1} \Hom_{\mathbb{S}_n}(\PL^{\ac}(n),\Com(n)) 
$ with the differential $[\alpha,-]$; the element $\alpha$ sends $s(a_1\triangleleft a_2)$ to $a_1\cdot a_2$. The $\mathbb{S}_n$-module isomorphism 
 $
\PL^{\ac}(n)\cong s^{n-1}\sgn_n\otimes\Perm(n)
 $
implies that
 \[
\prod_{n\ge 1} \Hom_{\mathbb{S}_n}(\PL^{\ac}(n),\Com(n))\cong \prod_{n\ge 1}s^{1-n}\Hom_{\mathbb{S}_n}(\Perm(n),\sgn_n). 
 \]
Since the $\mathbb{S}_n$-module $\Perm(n)$ is isomorphic to the standard permutation representation, it does not contain the sign representation for $n\ge 3$, and the respective Hom-spaces (computing the multiplicities of the sign representations) vanish. It follows that the convolution Lie algebra is concentrated in degrees zero and $-1$; the differential of the class of degree $0$ kills the class of degree~$-1$.  
\end{proof}

Using Theorem \ref{th:Koszul}, we obtain the Koszul dual result.

\begin{corollary}\label{cor:PLCom}
The deformation complex of the map $\Lie\to\Perm$ is acyclic.
\end{corollary}

\subsection{The map from the Lie operad}

Our result of this section is the first slightly non-trivial rigidity theorem involving the pre-Lie operad. A similar but much simpler argument shows that the deformation complex of the analogous map from the Lie operad to the associative operad is also acyclic.

\begin{theorem}\label{th:LiePreLie}
The deformation complex of the inclusion $\Lie\to\PL$ is acyclic.
\end{theorem}

\begin{proof}
Here, one uses the convolution algebra 
 $
\prod\limits_{n\ge 1} \Hom_{\mathbb{S}_n}(\Lie^{\ac}(n),\PL(n)) 
 $
with the differential $[\alpha,-]$; the element $\alpha$ sends $s[a_1,a_2]$ to $a_1\triangleleft a_2-a_2\triangleleft a_1$. Because of the isomorphisms
\begin{multline*}
\prod_{n\ge 1} \Hom_{\mathbb{S}_n}(\Lie^{\ac}(n),\PL(n)) \\ \cong \prod_{n\ge 1}\Hom_{\mathbb{S}_n}(s^{n-1}\sgn_n,\PL(n))\cong s\prod_{n\ge 2} \PL(n)\otimes_{\mathbb{S}_n} (\k s^{-1})^{\otimes n} , 
\end{multline*}
it is obvious that if we shift the homological degrees by one, that algebra may be identified with the underlying space of the free pre-Lie algebra generated by the element~$s^{-1}$. Elements of that space are combinations of unlabelled rooted trees (or, rather, rooted trees whose vertices are all labelled $s^{-1}$). The Maurer--Cartan element $\alpha$ in this case is the tree $\vcenter{
\xymatrix@M=3pt@R=7pt@C=5pt{
*+[o][F**]{}\ar@{-}[d]  \\
*+[o][F**]{} 
}}
$. The differential $d_\alpha$ is the ``usual'' graph complex differential \cite{MR1247289}: the image of each tree~$T$ is obtained as a sum of two terms:

-- the sum over all vertices of $T$ of all possible ways to split that vertex into two, and to connect the incoming edges of that vertex to one of the two vertices thus obtained, taken with the plus sign (corresponding to operadic insertions of the Maurer--Cartan element at vertices of~$T$):
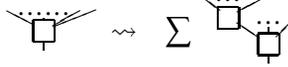
 \[
\vcenter{
\xymatrix@M=2pt@R=2pt@C=2pt{
\ar@{-}[dr]&\ldots\ar@{-}[d]\ldots&\ar@{-}[dl]&\ar@{-}[dll]\\
&*++[o][F**]{}\ar@{-}[d]&&\\
&&
}}
\leadsto\quad \sum
\vcenter{
\xymatrix@M=2pt@R=2pt@C=2pt{
\ar@{-}[dr]&\ldots&\ar@{-}[dl]\, \, &\\
&*++[o][F**]{}\ar@{-}[dr]&\ldots\ar@{-}[d]&\ar@{-}[dl]&\\
&&*++[o][F**]{}\ar@{-}[d]&&&&\\
&&
}}
 \] 

-- grafting the tree $T$ at the new root, taken with the minus sign, and the sum of all possible ways to create one extra black leaf, taken with the plus sign (corresponding to operadic insertions of $T$ at vertices of the Maurer--Cartan element).

The acyclicity of this complex can be proved by the following version of an argument of Willwacher \cite[Prop.~3.4]{MR3348138}. Let us call an \emph{antenna} of a tree $T$ a maximal connected subtree consisting of vertices of valences $1$ and $2$ (in particular, each leaf of $T$ is included in its own antenna); a tree can be viewed as a ``core'' without vertices of valence $1$ with several antennas attached to it. Let us consider the filtration of our complex by the size of the core. The first differential of the corresponding spectral sequence is a sum of terms adding edges to individual antennas, and can be easily seen acyclic. Indeed, we may first consider planar rooted trees, fixing a total order on the children of each vertex; this forces the graph automorphisms to disappear, and the first page of our spectral sequence is the tensor product of acyclic complexes for (the non-empty set of) individual antennas; dealing with trees without a fixed planar structure amounts to taking the subcomplex of invariants with respect to graph automorphisms, which is acyclic since it splits as a direct summand due to the Maschke theorem. 
\end{proof}

Using Theorem \ref{th:Koszul}, we obtain the Koszul dual result.

\begin{corollary}\label{cor:LiePL}
The deformation complex of the quotient map $\Perm\to\Com$ is acyclic.
\end{corollary}

\subsection{The map to the associative operad}

The next result can be viewed as a toy model of a deeper result proved in Section \ref{sec:defPL}.

\begin{theorem}\label{th:PLAss}
The deformation complex of the quotient map $\PL\to\Ass$ is acyclic.
\end{theorem}

\begin{proof}
Here, one uses the convolution algebra  
 $
\prod\limits_{n\ge 1} \Hom_{\mathbb{S}_n}(\PL^{\ac}(n),\Ass(n)) 
 $
 with the differential $[\alpha,-]$; the element $\alpha$ sends $s(a_1\triangleleft a_2)$ to $a_1\cdot a_2$. The $\mathbb{S}_n$-module isomorphism
 $
\PL^{\ac}(n)\cong s^{n-1}\sgn_n\otimes\Perm(n)
 $ 
implies that
 \[
\prod_{n\ge 1} \Hom_{\mathbb{S}_n}(\PL^{\ac}(n),\Ass(n))\cong s\prod_{n\ge 1} (\Ass(n)\otimes\Perm(n))\otimes_{\mathbb{S}_n} (\k s^{-1})^{\otimes n} ,
 \]
so if we shift the homological degrees by one, this convolution Lie algebra may be identified with the underlying space of the free associative dialgebra generated by the element~$s^{-1}$. The Maurer--Cartan element $\alpha$ in this case is binary product $\underline{a_1}a_2$ with the first element underlined. If we denote by $e_n^i$ the product of $n$ copies of $s^{-1}$ where the $i$-th factor is underlined, we have
\begin{align*}
d(e_n^i)&=
e_{n+1}^i+(-1)^{n-1}e_{n+1}^1-\sum_{k=1}^{i-1}(-1)^{k-1}e_{n+1}^{i+1}-\sum_{k=i}^{n}(-1)^{k-1}e_{n+1}^{i}\\
&=e_{n+1}^i+(-1)^{n-1}e_{n+1}^1-\frac{1+(-1)^i}{2}e_{n+1}^{i+1}+(-1)^{i}\frac{1+(-1)^{n-i}}{2}e_{n+1}^{i},
\end{align*}
or in other words,
 \[
d(e_n^i)=
\begin{cases}
e_{n+1}^1, \qquad\qquad\qquad\quad\! i \text{ odd }, n \text{ odd },\\
e_{n+1}^i-e_{n+1}^1,\qquad\qquad\, i \text{ odd }, n \text{ even }, \\
e_{n+1}^i+e_{n+1}^1-e_{n+1}^{i+1}, \quad\,\, i \text{ even }, n \text{ odd },\\
-e_{n+1}^{i+1}-e_{n+1}^1, \qquad\quad\,\, i \text{ even }, n \text{ even } .
\end{cases}
 \]
We note that the span of all $e_n^1$ is an acyclic subcomplex (it is in fact isomorphic to the deformation complex of the map $\Lie\to\Ass$), and the quotient by that subcomplex is also acyclic. 
\end{proof}

Using Theorem \ref{th:Koszul}, we obtain the Koszul dual result.

\begin{corollary}\label{cor:PLAss}
The deformation complex of the quotient map $\Ass\to\Perm$ is acyclic.
\end{corollary}

\section{Deformation complex of the pre-Lie operad}\label{sec:defPL}

In this section, we use intuition coming from examples of Section \ref{sec:maps} to prove the first ``serious'' rigidity theorem: the acyclicity of the deformation complex of the operad $\PL$. As mentioned in the introduction, this result implies that the group of homotopy automorphisms of the operad $\PL$ is the action of $\k^\times$ by intrinsic automorphisms: the element $\lambda\in\k^\times$ rescales each component $\PL(n)$ by~$\lambda^{n-1}$.

\begin{theorem}
The deformation complex of the identity map $\PL\to\PL$ is acyclic. 
\end{theorem}

\begin{proof}
Here, one works with the convolution algebra 
$
\prod\limits_{n\ge 1} \Hom_{\mathbb{S}_n}(\PL^{\ac}(n),\PL(n)) 
$
equipped with the differential $[\alpha,-]$; the element $\alpha$ sends $s(a_1\triangleleft a_2)$ to $a_1\triangleleft a_2$. Arguing as in the proof of Theorem \ref{th:PLAss}, we see that
 \[
\prod_{n\ge 1}\Hom_{\mathbb{S}_n}(\PL^{\ac}(n),\PL(n))\cong s\prod_{n\ge 1} (\PL(n)\otimes\Perm(n))\otimes_{\mathbb{S}_n} (\k s^{-1})^{\otimes n} ,
 \]
so if we shift the homological degrees by one, this convolution Lie algebra may be identified with the underlying space of the free pre-Lie dialgebra generated by the element~$s^{-1}$. For the reader who prefers a more combinatorial viewpoint, we would like to indicate that on the level of species, the set $\RT(n)\otimes\Perm(n)$ represents Joyal's vertebrates on $n$ labelled vertices \cite{MR633783}. Elements of that space are combinations of unlabelled rooted trees (or, rather, rooted trees whose vertices are all labelled $s^{-1}$) where one of the vertices (root or non-root) is distinguished; we call that vertex ``special'' and other vertices ``normal''. The Maurer--Cartan element $\alpha$ in this case is the tree $\vcenter{
\xymatrix@M=3pt@R=7pt@C=5pt{
*+[o][F**]{}\ar@{-}[d]  \\
*+[o][F-]{} 
}}
$ , where the black vertex is normal, and the white vertex is special: the distinguished vertex of the tree encoding the identity map of the operad $\PL$ is the root vertex. The differential $d_\alpha$ is similar to the usual graph complex differential: the image of each tree~$T$ is obtained as a sum of three terms:

-- the sum over all normal vertices of $T$ of all possible ways to split that vertex into two normal ones, and to connect the incoming edges of that vertex to one of the two vertices thus obtained, taken with the plus sign (corresponding to operadic insertions of the Maurer--Cartan element at normal vertices of~$T$):
 \[
\vcenter{
\xymatrix@M=2pt@R=2pt@C=2pt{
\ar@{-}[dr]&\ldots\ar@{-}[d]\ldots&\ar@{-}[dl]&\ar@{-}[dll]\\
&*++[o][F**]{}\ar@{-}[d]&&\\
&&
}}
\leadsto\quad \sum
\vcenter{
\xymatrix@M=2pt@R=2pt@C=2pt{
\ar@{-}[dr]&\ldots&\ar@{-}[dl]\, \, &\\
&*++[o][F**]{}\ar@{-}[dr]&\ldots\ar@{-}[d]&\ar@{-}[dl]&\\
&&*++[o][F**]{}\ar@{-}[d]&&&&\\
&&
}}
 \]

-- the sum over all possible ways to split the special vertex into two, and to connect the incoming edges of that vertex to one of the two vertices thus obtained, so that the one of the two vertices that is closer to root becomes the special vertex of the new tree (corresponding to the operadic insertion of the Maurer--Cartan element at the special vertex):
 \[
\vcenter
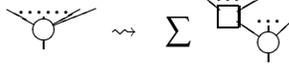
{
\xymatrix@M=2pt@R=2pt@C=2pt{
\ar@{-}[dr]&\ldots\ar@{-}[d]\ldots&\ar@{-}[dl]&\ar@{-}[dll]\\
&*++[o][F-]{}\ar@{-}[d]&\\
&&
}}
\leadsto\quad \sum
\vcenter{
\xymatrix@M=2pt@R=2pt@C=2pt{
\ar@{-}[dr]&\ldots&\ar@{-}[dl]\, \, &\\
&*++[o][F**]{}\ar@{-}[dr]&\ldots\ar@{-}[d]&\ar@{-}[dl]&\\
&&*++[o][F-]{}\ar@{-}[d]&&&&\\
&&
}}
 \]	

-- grafting the tree $T$ at the new root that becomes the special vertex in the new tree, taken with the minus sign, and the sum of all possible ways to create one extra normal leaf, taken with the plus sign (corresponding to operadic insertions of~$T$ at vertices of the Maurer--Cartan element). 

To compute the homology of this complex, we shall argue in two steps. First, let us consider the space spanned by all trees whose special vertex is the root vertex (the degenerate vertebrates of \cite{MR633783}). It is clear that this space is a subcomplex, and forgetting about the speciality of the root identifies this subcomplex with the standard graph complex discussed in the proof of Theorem~\ref{th:LiePreLie}. Thus, it is acyclic.
The quotient by this subcomplex is spanned by the trees whose root vertex is normal. Let us consider the \emph{spine} of such a tree defined as the path connecting the root to the special vertex. Clearly, the differential is made of terms that preserve the length of the spine and terms that increase it. We may consider the filtration for which the associated graded differential preserves the length of the spine. That associated graded complex splits as a sum of complexes with the spine of given length $m$, and such summand is the $m$-fold tensor product of complexes corresponding to the individual trees attached along the vertices of the spine (this is a homological version of the relationship about vertebrates and rooted trees \cite[Ex.~9]{MR633783}). Each factor attached at a normal vertex is the standard graph complex discussed in the proof of Theorem~\ref{th:LiePreLie}, and since the root vertex is normal, there is at least one such factor. It remains to use the K\"unneth formula to complete the proof. 
\end{proof}

Using Theorem \ref{th:Koszul}, we obtain the Koszul dual result.

\begin{corollary}\label{cor:DefPerm}
The deformation complex of the identity map $\Perm\to\Perm$ is acyclic.
\end{corollary}

\section{Twisting of the pre-Lie operad}\label{sec:twisting}

General operadic twisting was defined by Willwacher \cite[Appendix~I]{MR3348138} to work with Kontsevich's graph complexes; it is an endofunctor of the category of differential graded operads equipped with a morphism from the operad of shifted $L_\infty$-algebras. There exists a counterpart of that endofunctor for operads equipped with a morphism from the operad~$L_\infty$, see \cite[Sec.~3.5]{MR3299688}. In particular, the general definition can be applied to a dg operad $(\calP,d_{\calP})$ equipped with a map of dg operads 
 $
f\colon(\Lie,0)\to(\calP,d_{\calP})
 $
that sends the generator of $\Lie$ to a certain binary operation of $\calP$ that we denote $[-,-]$. In this case, the operad $\Tw(\calP)$, the result of applying the twisting procedure to the operad $\calP$, is a differential graded operad that can be defined as follows~\cite{dotsenko2018twisting}. Denote by $\alpha$ a new operation of arity~$0$ and degree~$-1$. The underlying non-differential operad of $\Tw(\calP)$ is the completed coproduct $\widehat{\calP\vee\k\alpha}$. To define the differential, one performs two steps. First, one considers the operad
 $
\MC(\calP)=\left(\widehat{\calP\vee\k\alpha}, d_{\calP}+d_{\MC} \right)
 $
encoding dg $\calP$-algebras with a Maurer--Cartan element; its differential is the sum of the differential of $\calP$ and the differential $d_{\MC}$ which vanishes on $\calP$ and satisfies $d_{\MC}(\alpha)=-\frac12[\alpha,\alpha]$. The element $\ell_1^{\alpha}\in \MC(\calP)(1)$ defined by the formula
$
\ell_1^{\alpha}(a_1)=[\alpha,a_1] 
$
is an operadic Maurer--Cartan element of $\MC(\calP)$ which we can use to twist the differential of that operad. One puts 
 \[
\Tw(\calP)=\left(\widehat{\calP\vee\k\alpha}, d_{\Tw}=d_{\calP}+d_{\MC}+[\ell_1^{\alpha},-] \right).
 \]
The reason to be interested in the operad $\Tw(\calP)$ is the following. Suppose that $A$ is a dg $\calP$-algebra, and suppose that $\alpha$ is a Maurer--Cartan element of the algebra $A$ viewed as a Lie algebra.
As discussed in Section \ref{sec:DefKoszul}, one can twist the differential of $A$; the twisted dg Lie algebra $(A,d_\alpha)$ can in fact be extended to a $\Tw(\calP)$-algebra structure. It is worth noting that for the operation $[-,-]$ inside $\calP\subset\Tw(\calP)$ we have $d_{\Tw}([-,-])=0$ since the operation $[-,-]$ is annihilated by $d_{\calP}$ and satisfies the Jacobi identity. This means that there is a map of dg operads from $(\Lie,0)$ to $\Tw(\calP)$, so each $\Tw(\calP)$-algebra has a functorial dg Lie algebra structure. 

Let us remark that for each operad $\calP$ with zero differential, all operations in the image of $d_{\Tw}$ contain at least one occurrence of $\alpha$. Thus, the image of $\Lie$ in $\Tw(\calP)$ on the level of cohomology is isomorphic to the image of the map $f\colon\Lie\to\calP$. In case of the operad $\PL$, the map $\Lie\to\PL$ is injective (even the composite $\Lie\to\PL\to\Ass$ is injective) so the inclusion of dg operads $\Lie\rightarrow\Tw(\PL)$ is injective on the level of homology. We shall now prove that the homology of the operad $\Tw(\PL)$ is exhausted by the image of that inclusion. 

\begin{theorem}\label{th:twisting}
The inclusion of dg operads $(\Lie,0)\rightarrow\Tw(\PL)$ induces a homology isomorphism. 
\end{theorem}

\begin{proof}
This result is close to~\cite[Th.~3.6]{willwacher2017prelie}; its proof, like the one in \emph{loc. cit.}, mimics \cite[Lemma~8.5]{MR3220287}. We shall examine the differential more carefully and then argue by induction on arity. The arity $n$ component $\Tw(\PL)(n)$ is spanned by rooted trees with ``normal'' vertices labelled $1$, $\ldots$, $n$ and a certain number of ``special'' vertices labelled~$\alpha$. The differential in $\Tw(\PL)$ is similar to the usual graph complex differential: the image of each tree~$T$ is obtained as a sum of three terms:

-- the sum over all possible ways to split a normal vertex into a normal one retaining the label and a special one, and to connect the incoming edges of that vertex to one of the two vertices thus obtained, so that the term where the vertex further from the root retains the label is taken with the plus sign, and the other term is taken with the minus sign (corresponding to the operadic insertions of $\ell_1^{\alpha}$ at labelled vertices):
 \[
\vcenter{
\xymatrix@M=2pt@R=2pt@C=2pt{
\ar@{-}[dr]&\ldots\ar@{-}[d]\ldots&\ar@{-}[dl]&\ar@{-}[dll]\\
&*++[o][F-]{s}\ar@{-}[d]&\\
&&
}}
\leadsto\quad \sum
\vcenter{
\xymatrix@M=2pt@R=2pt@C=2pt{
\ar@{-}[dr]&\ldots&\ar@{-}[dl]\, \, &\\
&*++[o][F-]{s}\ar@{-}[dr]&\ldots\ar@{-}[d]&\ar@{-}[dl]&\\
&&*++[o][F-]{\alpha}\ar@{-}[d]&&&&\\
&&
}}\hspace{-6mm}
-\vcenter{
\xymatrix@M=2pt@R=2pt@C=2pt{
\ar@{-}[dr]&\ldots&\ar@{-}[dl]\, \, &\\
&*++[o][F-]{\alpha}\ar@{-}[dr]&\ldots\ar@{-}[d]&\ar@{-}[dl]&\\
&&*++[o][F-]{s}\ar@{-}[d]&&&&\\
&&
}}
 \]	

-- the sum over all special vertices of $T$ of all possible ways to split that vertex into two special ones, and to connect the incoming edges of that vertex to one of the two vertices thus obtained, taken with the plus sign (corresponding to computing the differential of the Maurer--Cartan element, since we have $d_{\Tw}(\alpha)=d(\alpha)+\ell_1^\alpha(\alpha)=-\frac12[\alpha,\alpha]+[\alpha,\alpha]=\frac12[\alpha,\alpha]=\alpha\triangleleft\alpha$): 
 \[
\vcenter{
\xymatrix@M=2pt@R=2pt@C=2pt{
\ar@{-}[dr]&\ldots\ar@{-}[d]\ldots&\ar@{-}[dl]&\ar@{-}[dll]\\
&*++[o][F-]{\alpha}\ar@{-}[d]&&\\
&&
}}
\leadsto\quad \sum
\vcenter{
\xymatrix@M=2pt@R=2pt@C=2pt{
\ar@{-}[dr]&\ldots&\ar@{-}[dl]\, \, &\\
&*++[o][F-]{\alpha}\ar@{-}[dr]&\ldots\ar@{-}[d]&\ar@{-}[dl]&\\
&&*++[o][F-]{\alpha}\ar@{-}[d]&&&&\\
&&
}}
 \]

-- grafting the tree $T$ at the new special root, taken with the minus sign, and the sum of all possible ways to create one extra special leaf, taken with the plus sign (corresponding to operadic insertions of the tree $T$ at the only vertex of $\ell_1^{\alpha}$).

In particular, we note that $\Tw(\PL)(0)$ is, as a complex, isomorphic to the deformation complex of the inclusion $\Lie\to\PL$ from Theorem \ref{th:LiePreLie}, and as such is acyclic. Consider some arity $n>0$, and the decomposition 
 $
\Tw(\PL)(n)=V(n)\oplus W(n),
 $ 
 where $V(n)$ is spanned by the trees where the normal vertex with label~$1$ has less than two incident edges, and $W(n)$ is spanned by the trees where the normal vertex with label~$1$ has at least two incident edges. The differential has components mapping $V(n)$ to $V(n)$, mapping $W(n)$ to $V(n)$, and mapping $W(n)$ to $W(n)$. We consider the filtration $F^\bullet\Tw(\PL)(n)$ for which $F^pV(n)$ is spanned by trees from $V(n)$ with at least $p$ edges and $F^pW(n)$ is spanned by trees from $W(n)$ with at least $p+1$ edges. In the associated spectral sequence, the first differential is merely the part of the full differential that maps $W(n)$ to $V(n)$; it takes the normal vertex labelled $1$ in~$T$, makes this vertex special, and creates a new normal univalent vertex labelled~$1$ that is connected to~$v$:
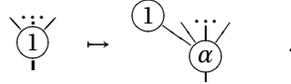
 \[
\vcenter{
\xymatrix@M=2pt@R=2pt@C=2pt{
\ar@{-}[dr]&\ldots\ar@{-}[d]&\ar@{-}[dl]\\
&*++[o][F-]{1}\ar@{-}[d]&&\\
&&
}}
\mapsto 
\vcenter{
\xymatrix@M=2pt@R=2pt@C=2pt{
&*++[o][F-]{1}\ar@{-}[drr]&\ar@{-}[dr]&\ldots\ar@{-}[d]&\ar@{-}[dl]&\\
&&&*++[o][F-]{\alpha}\ar@{-}[d]&&&&\\
&&&
}} .
 \]
This map is clearly injective. For $n=1$ its cokernel is spanned by the single one-vertex tree with its normal vertex labelled~$1$, proving that $\Lie(1)\to\Tw(\PL)(1)$ is a quasi-isomorphism. For $n>1$, the cokernel is spanned by the trees $T$ for which the normal vertex labelled $1$ is univalent and connected to another normal one.  It can be thus split into a direct sum of subcomplexes according to the number $k$ of that latter normal vertex; the number of such subcomplexes in arity $n$ is equal to $n-1$. We may proceed by induction by erasing the normal vertex labelled $1$: each of these subcomplexes is assumed to have homology $\Lie(n-1)$ of dimension $(n-2)!$, so the total dimension of homology of $\Tw(\PL)$ in arity $n$ is $(n-1)!$ which is the same as the dimension of $\Lie(n)$, implying that the inclusion of the operad $\Lie$ in $H_\bullet(\Tw(\PL))$ is an isomorphism. 
\end{proof}	

\section{Applications and further results}

\subsection{Lie elements in pre-Lie algebras}\label{sec:LPL}

We shall now recall Markl's criterion for Lie elements in the free pre-Lie algebra $\PL(V)$ generated by a vector space $V$~\cite{MR2325698}, and discuss its relationship to Theorem \ref{th:twisting}. To state that criterion, one builds the pre-Lie algebra $\rPL(V)$ (where `r' stands for `reduced') defined by the formula
 \[
\rPL(V):=\frac{\PL(V\oplus\k\circ)}{(\circ\triangleleft\circ)},
 \]
where $\circ$ is an additional generator of degree $-1$. According to \cite[Prop.~3.2]{MR2325698}, there is a well defined map $d\colon\rPL(V)\to\rPL(V)$ of degree~$-1$ that annihilates all generators and satisfies $d^2=0$, but, rather than making $\rPL(V)$ a dg pre-Lie algebra, the condition
 \[
d(a\triangleleft b)=d(a)\triangleleft b+(-1)^{|a|} a\triangleleft d(b)+Q(a,b) 
 \]
holds for $Q(a,b)=(\circ\triangleleft a)\triangleleft b-\circ\triangleleft (a\triangleleft b)$. 

\begin{theorem}[{\cite[Th.~3.3]{MR2325698}}]
The subspace of Lie elements in $\PL(V)$ equals the kernel of $d$ on the space of degree~$0$ elements $\rPL(V)_0\cong\PL(V)$:
 \[
\Lie(V)\cong \ker(d\colon\rPL(V)_0\to\rPL(V)_{-1}).
 \]
\end{theorem}

In fact, as explained in \emph{loc. cit.}, one can define a differential graded operad $\rPL$ and view the chain complex $(\rPL(V),d)$ as the result of evaluating the Schur functor corresponding to differential graded $\mathbb{S}$-module $\rPL$ on the vector space $V$. Armed with this understanding and recalling that $d(a\triangleleft b)-d(a)\triangleleft b-(-1)^{|a|} a\triangleleft d(b)$ is simply the operadic differential $\partial(-\triangleleft-)=[d,-\triangleleft-]$ evaluated on $a\otimes b$, we infer that the differential of the dg operad $\rPL$ is fully defined by its value on the pre-Lie product $(-\triangleleft-)$ which is equal to $Q(-,-)$. We now note that in terms of rooted trees one has
 \[
Q(a,b)=(\circ\triangleleft a)\triangleleft b-\circ\triangleleft (a\triangleleft b)=
\vcenter{
\xymatrix@M=3pt@C=5pt@R=10pt{
*+[o][F-]{a}\ar@{-}[dr] &&  *+[o][F-]{b}\ar@{-}[dl] \\
& *+[o][F-]{} & 
}} .
 \]
This brings us very close to the key revelation: in the operad $\Tw(\PL)$, we have
 \[
d_{\Tw}\left(\,\,\vcenter{
\xymatrix@M=3pt@R=7pt@C=5pt{
*+[o][F-]{2}\ar@{-}[d]  \\
*+[o][F-]{1} 
}}\,\,
\right) = 
\vcenter{
\xymatrix@M=3pt@C=5pt@R=10pt{
*+[o][F-]{1}\ar@{-}[dr] &&  *+[o][F-]{2}\ar@{-}[dl] \\
& *+[o][F-]{\alpha} & 
}} ,
 \]
so one must think of the element $\circ$ as a shadow of the element $\alpha\in\Tw(\PL)$. Now we are ready to make the connection precise. We note that the operadic ideal of $\Tw(\PL)$ generated by $\alpha\triangleleft\alpha$ is closed under differential, and so one may consider the filtration by powers of that ideal and the associated graded chain complex. An argument identical to that in the proof of Theorem~\ref{th:twisting} can be used to show that the homology of the associated graded complex is already isomorphic to the operad $\Lie$. In particular, one can prove the following result, solving Markl's problem \cite[Problem~7.6]{MR2325698}. 

\begin{theorem}
The quotient dg operad $\Tw(\PL)/(\alpha\triangleleft\alpha)$ is isomorphic to the dg operad $\rPL$, and the morphism
$(\Lie,0)\to \rPL$
induces a homology isomorphism.
\end{theorem}

\subsection{Natural operations on deformation complexes}

As we mentioned earlier, one motivation for studying the operad $\Tw(\calP)$ comes from the search of operations that are naturally defined on twisted dg Lie algebras coming from dg $\calP$-algebras. Since deformation complexes of operad maps discussed in Section \ref{sec:DefKoszul} arise from pre-Lie algebras, the twisted dg Lie algebra structure on each of them is in fact a part of a $\Tw(\PL)$-algebra structure. Theorem \ref{th:twisting} shows that the dg Lie algebra structure is the only homotopy invariant structure one may define \emph{from this viewpoint}. It is natural to ask if there are other homotopy invariant operations on deformation complexes that one may exhibit. A general framework to studying natural operations on the deformation complex of the identity map of an algebra over a Koszul operad was proposed by Markl in \cite{MR2309979}, who defined the most general dg operad $\calB_\calP$ acting on such a complex. That operad is not very manageable in general, as it is spanned by operations indexed by certain decorated trees labelled additionally by elements 
 \[
\Phi_{a_1,\ldots,a_n,a}\in\calP^{\ac}(a_1)\otimes\cdots\otimes\calP^{\ac}(a_n)\otimes\calP^{\ac}(a)
 \]
for appropriate values of arities $a_1,\ldots,a_n,a$. 

There are exactly two situations where the definition of the operad $\calB_\calP$ simplifies rather drastically: the cases $\calP=\Lie$ and $\calP=\Ass$. In each of these cases, the components of $\calP^{\ac}$ are, in sense, one-dimensional: for the Lie operad, that is literally true, while for the associative operad, the components of the Koszul dual are free modules of rank one over the group algebras of symmetric groups, and the symmetric group action can be ``hidden'' in the decorations by means of making the trees planar. An exhaustive study of the operad $\calB_{\Ass}$ was undertaken by Batanin and Markl \cite{MR3261598} who established that this operad has the homotopy type of the operad of singular chains on the little disks operad, thus establishing what perhaps is the strongest form of the celebrated Deligne conjecture stating that the Gerstenhaber algebra structure of the Hochschild cohomology comes from the action of the operad of singular chains on the little disk operad on the level of the deformation complex. By contrast, the operad $\calB_{\Lie}$ has not been carefully studied before; to the best of our knowledge, the only situation where that operad was used is the work of Markl on Lie elements in free pre-Lie algebras. Markl conjectured \cite[Conj.~7]{MR2309979} that the operad $\calB_{\Lie}$ has the homotopy type of the operad $\Lie$; we shall now outline a proof of this conjecture. 

\begin{theorem}[Lie version of the Deligne conjecture]
The operad of natural operations on deformation complexes of Lie algebras has the homotopy type of the operad $\Lie$.
\end{theorem}

\begin{proof}
First, as noted in \cite{MR2309979}, the intrinsic formality theorem \cite[Prop.~3.4]{MR1380606} guarantees that it is enough to establish that $H_\bullet(\calB_{\Lie})=\Lie$. Second, it is easy to extend the definition of the operad $\calB_\calP$ to include the case where $\calP$ is a model of a Koszul operad, that is the cobar complex of a Koszul cooperad $\calP^{\ac}$. Examining the operad $\calB_{\Lie_\infty}$, we see that it is very close to the operad $\Tw(\PL)$; in fact, it is exactly the version of the operad $\Tw(\PL)$ where only trees whose special vertices have at least two children are allowed; it was introduced by Willwacher in \cite[Sec.~3.2]{willwacher2017prelie}. According to \cite[Th.~3.6]{willwacher2017prelie} which we already mentioned in the proof of Theorem \ref{th:twisting}, the homology of the latter operad is isomorphic to the operad Lie. It remains to show that the operads $\calB_{\Lie}$ and $\calB_{\Lie_\infty}$ are quasi-isomorphic. The proof of this result repeats \emph{mutatis mutandis} that of the theorem asserting that operadic twisting itself preserves quasi-isomorphisms, see~\cite[Th.~5.1]{MR3299688}. 
\end{proof}

\subsection{Planar trees and the brace operad}

A natural companion of the pre-Lie operad from the combinatorial viewpoint is the operad $\Br$ of brace algebras discovered independently by Getzler \cite{MR1261901}, Kadeishvili \cite{MR1029003}, and Ronco \cite{MR1800716}. Algebraically, it is an operad generated by infinitely many operations $\{a_0; a_1,\ldots, a_n\}$,  $n\ge 1$, 
satisfying a rather complicated system of identities. However, from the combinatorial point of view, one can realise that operad using labelled \emph{planar} rooted trees with operadic insertions described by a planar analogue of the Chapoton--Livernet formula, see \cite{MR1879927,MR1909461}. Let us record here two analogues of our results that can be proved by filtration arguments for graph complexes made of planar trees.

\begin{theorem}
The following complexes are acyclic:
\begin{enumerate}
\item the deformation complex of the map $\Lie\to\Br$ sending $[a_1,a_2]$ to $\{a_1;a_2\}-\{a_2;a_1\}$,
\item the deformation complex of the map $\PL\to\Br$ sending $a_1\triangleleft a_2$ to $\{a_1;a_2\}$.
\end{enumerate}
\end{theorem}

We note that deformaton theory of maps from the brace operad is much harder to study since that operad is not quadratic and therefore not Koszul. One however may apply the operadic twisting to that operad; the result is much more complicated than that for the pre-Lie operad. In fact, according to \cite[Th.~9.3]{MR3299688}, the differential graded operad $\Tw(\Br)$ is quasi-isomorphic to the differential graded brace operad that prominently features in various proofs of the Deligne conjecture \cite{MR1321701,MR1805894,MR1890736,MR1328534}, which allows to compute the homology of the operad $\Tw(\Br)$. 

\begin{theorem}[{\cite{MR3299688}}]
We have the operad isomorphism $H_\bullet(\Tw(\Br))\cong\calS\Gerst$.
\end{theorem}

In particular, one has $H_0(\Tw(\Br))\cong\Lie$, and, as indicated in \cite[Sec.~1.4]{MR2325698}, one can use it to describe Lie elements in free brace algebras. It is natural to ask for which Hopf cooperads $\calC$ we have $H_0(\Tw(\PL_\calC))\cong\Lie$,
where $\PL_\calC$ is the operad of $\calC$-enriched rooted trees \cite{MR3411136,dotsenko2020operads} which coincides with $\PL$ for $\calC=\uCom^*$ and with $\Br$ for $\calC=\uAss^*$. This question should be compared to a conjecture of Markl \cite[Conjecture~22]{MR2309979} on the degree zero homology of the operad of natural operations $\calB_\calP$.

\section*{Acknowledgements } We thank Martin Markl for comments on a draft version of this paper. In preparation of the final version of the manuscript, research of the first author was supported by the project HighAGT ANR-20-CE40-0016. Research of the second author was carried out within the HSE University Basic Research Program and supported in part by the Russian Academic Excellence Project '5-100' and in part by the Simons Foundation. This work started during the second author's visit to Trinity College Dublin which became possible because of the financial support of Visiting Professorships and Fellowships Benefaction Fund. Results of  Section \ref{sec:DefKoszul} (in particular, Theorem~\ref{th:Koszul}) have been obtained under support of the RSF grant No.19-11-00275.

\bibliographystyle{amsplain} 
\bibliography{PreLieHtpy.bib}

\end{document}